\documentclass{amsproc}
\usepackage{verbatim,amsmath,amscd,
amssymb,amsthm,color,graphics}
\def\m@th{\mathsurround=0pt}

\def\fsquare(#1,#2){
\hbox{\vrule$\hskip-0.4pt\vcenter to #1{\normalbaselines\m@th
\hrule\vfil\hbox to #1{\hfill$\scriptstyle #2$\hfill}\vfil\hrule}$\hskip-0.4pt
\vrule}}

\def\addsquare(#1,#2){\hbox{$
	\dimen1=#1 \advance\dimen1 by -0.8pt
	\vcenter to #1{\hrule height0.4pt depth0.0pt%
	\hbox to #1{%
	\vbox to \dimen1{\vss%
	\hbox to \dimen1{\hss$\scriptstyle~#2~$\hss}%
	\vss}%
	\vrule width0.4pt}%
	\hrule height0.4pt depth0.0pt}$}}

\def\Fsquare(#1,#2){
\hbox{\vrule$\hskip-0.4pt\vcenter to #1{\normalbaselines\m@th
\hrule\vfil\hbox to #1{\hfill$#2$\hfill}\vfil\hrule}$\hskip-0.4pt
\vrule}}

\def\Addsquare(#1,#2){\hbox{$
	\dimen1=#1 \advance\dimen1 by -0.8pt
	\vcenter to #1{\hrule height0.4pt depth0.0pt%
	\hbox to #1{%
	\vbox to \dimen1{\vss%
	\hbox to \dimen1{\hss$~#2~$\hss}%
	\vss}%
	\vrule width0.4pt}%
	\hrule height0.4pt depth0.0pt}$}}

\def\hfourbox(#1,#2,#3,#4){%
	\fsquare(0.3cm,#1)\addsquare(0.3cm,#2)\addsquare(0.3cm,#3)\addsquare(0.3cm,#4)}

\def\Hfourbox(#1,#2,#3,#4){%
	\Fsquare(0.4cm,#1)\Addsquare(0.4cm,#2)\Addsquare(0.4cm,#3)\Addsquare(0.4cm,#4)}

\def\HHfourbox(#1,#2,#3,#4){%
	\Fsquare(0.8cm,#1)\Addsquare(0.8cm,#2)\Addsquare(0.8cm,#3)\Addsquare(0.8cm,#4)}

\def\fsq(#1){%
          \fsquare(0.3cm,#1)}

\def\hthreebox(#1,#2,#3){%
	\fsquare(0.3cm,#1)\addsquare(0.3cm,#2)\addsquare(0.3cm,#3)}

\def\htwobox(#1,#2){%
	\fsquare(0.3cm,#1)\addsquare(0.3cm,#2)}

\def\vfourbox(#1,#2,#3,#4){%
	\normalbaselines\m@th\offinterlineskip
	\vcenter{\hbox{\fsquare(0.3cm,#1)}
	      \vskip-0.4pt
	      \hbox{\fsquare(0.3cm,#2)}	
	      \vskip-0.4pt
	      \hbox{\fsquare(0.3cm,#3)}	
	      \vskip-0.4pt
	      \hbox{\fsquare(0.3cm,#4)}}}

\def\VVfourbox(#1,#2,#3,#4){%
	\normalbaselines\m@th\offinterlineskip
	\vcenter{\hbox{\Fsquare(0.8cm,#1)}
	      \vskip-0.4pt
	      \hbox{\Fsquare(0.8cm,#2)}	
	      \vskip-0.4pt
	      \hbox{\Fsquare(0.8cm,#3)}	
	      \vskip-0.4pt
	      \hbox{\Fsquare(0.8cm,#4)}}}

\def\Vfourbox(#1,#2,#3,#4){%
	\normalbaselines\m@th\offinterlineskip
	\vcenter{\hbox{\Fsquare(0.4cm,#1)}
	      \vskip-0.4pt
	      \hbox{\Fsquare(0.4cm,#2)}	
	      \vskip-0.4pt
	      \hbox{\Fsquare(0.4cm,#3)}	
	      \vskip-0.4pt
	      \hbox{\Fsquare(0.4cm,#4)}}}

\def\vthreebox(#1,#2,#3){%
	\normalbaselines\m@th\offinterlineskip
	\vcenter{\hbox{\fsquare(0.3cm,#1)}
	      \vskip-0.4pt
	      \hbox{\fsquare(0.3cm,#2)}	
	      \vskip-0.4pt
	      \hbox{\fsquare(0.3cm,#3)}}}

\def\vtwobox(#1,#2){%
	\normalbaselines\m@th\offinterlineskip
	\vcenter{\sbox{\fsquare(0.3cm,#1)}
	      \vskip-0.4pt
	      \hbox{\fsquare(0.3cm,#2)}}}

\def\vtwobox2(#1,#2){%
	\normalbaselines\m@th\offinterlineskip
	\vcenter{\hbox{#1}
	      \vskip-0.4pt
	      \hbox{\fsquare(0.3cm,#2)}}}

\def\Hthreebox(#1,#2,#3){%
	\Fsquare(0.4cm,#1)\Addsquare(0.4cm,#2)\Addsquare(0.4cm,#3)}

\def\HHthreebox(#1,#2,#3){%
	\Fsquare(0.8cm,#1)\Addsquare(0.8cm,#2)\Addsquare(0.8cm,#3)}

\def\Htwobox(#1,#2){%
	\Fsquare(0.4cm,#1)\Addsquare(0.4cm,#2)}

\def\H6twobox(#1,#2){%
	\Fsquare(0.6cm,#1)\Addsquare(0.6cm,#2)}

\def\HHtwobox(#1,#2){%
	\Fsquare(0.8cm,#1)\Addsquare(0.8cm,#2)}

\def\Vthreebox(#1,#2,#3){%
	\normalbaselines\m@th\offinterlineskip
	\vcenter{\hbox{\Fsquare(0.4cm,#1)}
	      \vskip-0.4pt
	      \hbox{\Fsquare(0.4cm,#2)}	
	      \vskip-0.4pt
	      \hbox{\Fsquare(0.4cm,#3)}}}

\def\Vtwobox(#1,#2){%
	\normalbaselines\m@th\offinterlineskip
	\vcenter{\hbox{\Fsquare(0.4cm,#1)}
	      \vskip-0.4pt
	      \hbox{\Fsquare(0.4cm,#2)}}}

\def\twoone(#1,#2,#3){%
	\normalbaselines\m@th\offinterlineskip
	\vcenter{\hbox{\htwobox({#1},{#2})}
	      \vskip-0.4pt
	      \hbox{\fsquare(0.3cm,#3)}}}

\def\twothree(#1,#2,#3,#4,#5){%
	\normalbaselines\m@th\offinterlineskip
	\vcenter{\hbox{\htwobox({#1},{#2})}
	      \vskip-0.4pt
	      \hbox{\hthreebox({#3},{#4},{#5})}}}

\def\threethree(#1,#2,#3,#4,#5,#6){%
	\normalbaselines\m@th\offinterlineskip
	\vcenter{\hbox{\hthreebox({#1},{#2},{#3})}
	      \vskip-0.4pt
	      \hbox{\hthreebox({#4},{#5},{#6})}}}

\def\threethreeone(#1,#2,#3,#4,#5,#6,#7){%
	\normalbaselines\m@th\offinterlineskip
	\vcenter{\hbox{\hthreebox({#1},{#2},{#3})}
	      \vskip-0.4pt
	      \hbox{\hthreebox({#4},{#5},{#6})}
              \vskip-0.4pt
              \hbox{\fsquare(0.3cm,#7)}}}

\def\threethreetwo(#1,#2,#3,#4,#5,#6,#7,#8){%
	\normalbaselines\m@th\offinterlineskip
	\vcenter{\hbox{\hthreebox({#1},{#2},{#3})}
	      \vskip-0.4pt
	      \hbox{\hthreebox({#4},{#5},{#6})}
              \vskip-0.4pt
              \hbox{\htwobox({#7},{#8})}}}

\def\Twoone(#1,#2,#3){%
	\normalbaselines\m@th\offinterlineskip
	\vcenter{\hbox{\Htwobox({#1},{#2})}
	      \vskip-0.4pt
	      \hbox{\Fsquare(0.4cm,#3)}}}

\def\TTwoone(#1,#2,#3){%
	\normalbaselines\m@th\offinterlineskip
	\vcenter{\hbox{\H6twobox({#1},{#2})}
	      \vskip-0.4pt
	      \hbox{\Fsquare(0.6cm,#3)}}}

\def\threeone(#1,#2,#3,#4){%
	\normalbaselines\m@th\offinterlineskip
	\vcenter{\hbox{\hthreebox({#1},{#2},{#3})}
	      \vskip-0.4pt
	      \hbox{\fsquare(0.3cm,#4)}}}

\def\Threeone(#1,#2,#3,#4){%
	\normalbaselines\m@th\offinterlineskip
	\vcenter{\hbox{\Hthreebox({#1},{#2},{#3})}
	      \vskip-0.4pt
	      \hbox{\Fsquare(0.4cm,#4)}}}

\def\Threetwo(#1,#2,#3,#4,#5){%
	\normalbaselines\m@th\offinterlineskip
	\vcenter{\hbox{\Hthreebox({#1},{#2},{#3})}
	      \vskip-0.4pt
	      \hbox{\Htwobox({#4},{#5})}}}

\def\threetwo(#1,#2,#3,#4,#5){%
	\normalbaselines\m@th\offinterlineskip
	\vcenter{\hbox{\hthreebox({#1},{#2},{#3})}
	      \vskip-0.4pt
	      \hbox{\htwobox({#4},{#5})}}}

\def\twotwo(#1,#2,#3,#4){%
	\normalbaselines\m@th\offinterlineskip
	\vcenter{\hbox{\htwobox({#1},{#2})}
	      \vskip-0.4pt
	      \hbox{\htwobox({#3},{#4})}}}

\def\Twotwo(#1,#2,#3,#4){%
	\normalbaselines\m@th\offinterlineskip
	\vcenter{\hbox{\Htwobox({#1},{#2})}
	      \vskip-0.4pt
	      \hbox{\Htwobox({#3},{#4})}}}

\def\TTwotwo(#1,#2,#3,#4){%
	\normalbaselines\m@th\offinterlineskip
	\vcenter{\hbox{\H6twobox({#1},{#2})}
	      \vskip-0.4pt
	      \hbox{\H6twobox({#3},{#4})}}}

\def\twooneone(#1,#2,#3,#4){%
	\normalbaselines\m@th\offinterlineskip
	\vcenter{\hbox{\htwobox({#1},{#2})}
	      \vskip-0.4pt
	      \hbox{\fsquare(0.3cm,#3)}
	      \vskip-0.4pt
	      \hbox{\fsquare(0.3cm,#4)}}}

\def\Twooneone(#1,#2,#3,#4){%
	\normalbaselines\m@th\offinterlineskip
	\vcenter{\hbox{\Htwobox({#1},{#2})}
	      \vskip-0.4pt
	      \hbox{\Fsquare(0.4cm,#3)}
	      \vskip-0.4pt
	      \hbox{\Fsquare(0.4cm,#4)}}}

\def\Twotwoone(#1,#2,#3,#4,#5){%
	\normalbaselines\m@th\offinterlineskip
	\vcenter{\hbox{\Htwobox({#1},{#2})}
	      \vskip-0.4pt
	      \hbox{\Htwobox({#3},{#4})}
              \vskip-0.4pt
	      \hbox{\Fsquare(0.4cm,#5)}}}

\def\twotwoone(#1,#2,#3,#4,#5){%
	\normalbaselines\m@th\offinterlineskip
	\vcenter{\hbox{\htwobox({#1},{#2})}
	      \vskip-0.4pt
	      \hbox{\htwobox({#3},{#4})}
              \vskip-0.4pt
	      \hbox{\fsquare(0.3cm,#5)}}}

\def\twotwotwo(#1,#2,#3,#4,#5,#6){%
	\normalbaselines\m@th\offinterlineskip
	\vcenter{\hbox{\htwobox({#1},{#2})}
	      \vskip-0.4pt
	      \hbox{\htwobox({#3},{#4})}
              \vskip-0.4pt
	      \hbox{\htwobox({#5},{#6})}}}

\def\threetwoone(#1,#2,#3,#4,#5,#6){%
	\normalbaselines\m@th\offinterlineskip
	\vcenter{\hbox{\hthreebox({#1},{#2},{#3})}
	      \vskip-0.4pt
	      \hbox{\htwobox({#4},{#5})}
              \vskip-0.4pt
	      \hbox{\fsquare(0.3cm,#6)}}}

\def\fourthreetwoone{%
	\normalbaselines\m@th\offinterlineskip
	\vcenter{\hbox{\hfourbox(,,,)}
	      \vskip-0.4pt
              \hbox{\hthreebox(,,)}
              \vskip-0.4pt
	      \hbox{\htwobox(,)}
              \vskip-0.4pt
	      \hbox{\fsquare(0.3cm,)}}}

\def\fourtwoone(#1,#2,#3,#4,#5,#6,#7){%
	\normalbaselines\m@th\offinterlineskip
	\vcenter{\hbox{\hfourbox({#1},{#2},{#3},{#4})}
	      \vskip-0.4pt
	      \hbox{\htwobox({#5},{#6})}
              \vskip-0.4pt
	      \hbox{\fsquare(0.3cm,{#7})}}}

\def\fourthreetwo(#1,#2,#3,#4,#5,#6,#7,#8,#9){%
	\normalbaselines\m@th\offinterlineskip
	\vcenter{\hbox{\hfourbox({#1},{#2},{#3},{#4})}
	      \vskip-0.4pt
              \hbox{\hthreebox({#5},{#6},{#7})}
              \vskip-0.4pt
	      \hbox{\htwobox({#8},{#9})}}}

\def\fourthreeone(#1,#2,#3,#4,#5,#6,#7,#8){%
	\normalbaselines\m@th\offinterlineskip
	\vcenter{\hbox{\hfourbox({#1},{#2},{#3},{#4})}
	      \vskip-0.4pt
              \hbox{\hthreebox({#5},{#6},{#7})}
              \vskip-0.4pt
	      \hbox{\fsquare(0.3cm,{#8})}}}

\def\TThreetwoone(#1,#2,#3,#4,#5,#6){%
	\normalbaselines\m@th\offinterlineskip
	\vcenter{\hbox{\HHthreebox({#1},{#2},{#3})}
	      \vskip-0.4pt
	      \hbox{\HHtwobox({#4},{#5})}
              \vskip-0.4pt
	      \hbox{\fsquare(0.8cm,#6)}}}

\def\Threeoneone(#1,#2,#3,#4,#5){%
	\normalbaselines\m@th\offinterlineskip
	\vcenter{\hbox{\Hthreebox({#1},{#2},{#3})}
	      \vskip-0.4pt
	      \hbox{\Fsquare(0.4cm,#4)}
              \vskip-0.4pt
	      \hbox{\Fsquare(0.4cm,#5)}}}

\def\threeoneone(#1,#2,#3,#4,#5){%
	\normalbaselines\m@th\offinterlineskip
	\vcenter{\hbox{\hthreebox({#1},{#2},{#3})}
	      \vskip-0.4pt
	      \hbox{\fsquare(0.3cm,#4)}
              \vskip-0.4pt
	      \hbox{\fsquare(0.3cm,#5)}}}

\def\FFourtwoone(#1,#2,#3,#4,#5,#6,#7){%
	\normalbaselines\m@th\offinterlineskip
	\vcenter{\hbox{\HHfourbox({#1},{#2},{#3},{#4})}
	      \vskip-0.4pt
	      \hbox{\HHtwobox({#5},{#6})}
              \vskip-0.4pt
	      \hbox{\fsquare(0.8cm,#7)}}}

\def\a{\fsquare(0.3cm){1}\addsquare(0.3cm)(2)\addsquare(0.3cm)(3)}

\def\b{\hbox{%
	\normalbaselines\m@th\offinterlineskip
	\vcenter{\hbox{\fsquare(0.3cm){2}}\vskip-0.4pt\hbox{\fsquare(0.3cm){2}}}}}

\def\c{\hbox{\normalbaselines\m@th\offinterlineskip%
	\vcenter{\hbox{\a}\vskip-0.4pt\hbox{\b}}}}

\def\ffsquare#1{%
	\fsquare(0.4cm,\hbox{#1})}

\def\naga{%
	\hbox{$\vcenter to 0.4cm{\normalbaselines\m@th
	\hrule\vfil\hbox to 1.2cm{\hfill$\cdots$\hfill}\vfil\hrule}$}}

\def\vnaga{\normalbaselines\m@th\baselineskip0pt\offinterlineskip%
	\vrule\vbox to 1.2cm{\vskip7pt\hbox to \dimen1{$\hfil\vdots\hfil$}\vfil}\vrule}

\def\dvbox{\hbox{\normalbaselines\m@th\baselineskip0pt\offinterlineskip\vbox{%
	  \hbox{$\ffsquare 1$}\vskip-0.4pt\hbox{$\vnaga$}\vskip-0.4pt\hbox{$\ffsquare N$}}}}

\def\sq(#1){\fsquare(0.4cm,#1)}
\def\Sq(#1){\fsquare(0.5cm,#1)}
\def\SSq(#1){\fsquare(0.9cm,#1)}

\def\mapright#1{\smash{\mathop{\longrightarrow}\limits^{#1}}}

\begin{document}

\title{Pictures and Littlewood-Richardson Crystals}
\author{Toshiki N\textsc{akashima}}
\address{\hspace{-20pt}T.N.:\,Department of Mathematics, Sophia University, 
Kioicho 7-1, Chiyoda-ku, Tokyo 102-8554, Japan}
\email{toshiki@mm.sophia.ac.jp,\,\,toshiki@sophia.ac.jp}
\author{Miki S\textsc{himojo}}
\address{\hspace{-20pt}
M.S.:\,Department of Mathematics, Sophia University, 
Kioicho 7-1, Chiyoda-ku, Tokyo 102-8554, Japan}
\email{m-shimoj@sophia.ac.jp}

\thanks{The first author is supported in part by  JSPS Grants in 
Aid for Scientific Research \#19540050.}

\date{April , 2009}


\begin{abstract}
We shall describe the one-to-one correspondence between 
the set of pictures and the set of Littlewood-Richardson crystals.
\end{abstract}

\maketitle

\newcommand{\Pic}{\mathbf{P}(\mu, \nu \setminus \lambda)}
\newcommand{\Cry}{\mathbf{B}(\mu)^{\nu}_{\lambda}}
\newcommand{\MECry}{\fsquare(0.3cm,i_1) \otimes \fsquare(0.3cm,i_2) \otimes \cdots \otimes  \fsquare(0.3cm,i_k) \otimes \cdots \otimes \fsquare(0.3cm,i_N)}

\renewcommand{\labelenumi}{$($\roman{enumi}$)$}
\renewcommand{\labelenumii}{$(${\rm \alph{enumii}}$)$}
\font\germ=eufm10

\newcommand{\cI}{{\mathcal I}}
\newcommand{\cA}{{\mathcal A}}
\newcommand{\cB}{{\mathcal B}}
\newcommand{\cC}{{\mathcal C}}
\newcommand{\cD}{{\mathcal D}}
\newcommand{\cF}{{\mathcal F}}
\newcommand{\cH}{{\mathcal H}}
\newcommand{\cK}{{\mathcal K}}
\newcommand{\cL}{{\mathcal L}}
\newcommand{\cM}{{\mathcal M}}
\newcommand{\cN}{{\mathcal N}}
\newcommand{\cO}{{\mathcal O}}
\newcommand{\cS}{{\mathcal S}}
\newcommand{\cV}{{\mathcal V}}
\newcommand{\fra}{\mathfrak a}
\newcommand{\frb}{\mathfrak b}
\newcommand{\frc}{\mathfrak c}
\newcommand{\frd}{\mathfrak d}
\newcommand{\fre}{\mathfrak e}
\newcommand{\frf}{\mathfrak f}
\newcommand{\frg}{\mathfrak g}
\newcommand{\frh}{\mathfrak h}
\newcommand{\fri}{\mathfrak i}
\newcommand{\frj}{\mathfrak j}
\newcommand{\frk}{\mathfrak k}
\newcommand{\frI}{\mathfrak I}
\newcommand{\fm}{\mathfrak m}
\newcommand{\frn}{\mathfrak n}
\newcommand{\frp}{\mathfrak p}
\newcommand{\fq}{\mathfrak q}
\newcommand{\frr}{\mathfrak r}
\newcommand{\frs}{\mathfrak s}
\newcommand{\frt}{\mathfrak t}
\newcommand{\fru}{\mathfrak u}
\newcommand{\frA}{\mathfrak A}
\newcommand{\frB}{\mathfrak B}
\newcommand{\frF}{\mathfrak F}
\newcommand{\frG}{\mathfrak G}
\newcommand{\frH}{\mathfrak H}
\newcommand{\frJ}{\mathfrak J}
\newcommand{\frN}{\mathfrak N}
\newcommand{\frP}{\mathfrak P}
\newcommand{\frT}{\mathfrak T}
\newcommand{\frU}{\mathfrak U}
\newcommand{\frV}{\mathfrak V}
\newcommand{\frX}{\mathfrak X}
\newcommand{\frY}{\mathfrak Y}
\newcommand{\frZ}{\mathfrak Z}
\newcommand{\rA}{\mathrm{A}}
\newcommand{\rC}{\mathrm{C}}
\newcommand{\rd}{\mathrm{d}}
\newcommand{\rB}{\mathrm{B}}
\newcommand{\rD}{\mathrm{D}}
\newcommand{\rE}{\mathrm{E}}
\newcommand{\rH}{\mathrm{H}}
\newcommand{\rK}{\mathrm{K}}
\newcommand{\rL}{\mathrm{L}}
\newcommand{\rM}{\mathrm{M}}
\newcommand{\rN}{\mathrm{N}}
\newcommand{\rR}{\mathrm{R}}
\newcommand{\rT}{\mathrm{T}}
\newcommand{\rZ}{\mathrm{Z}}
\newcommand{\bbA}{\mathbb A}
\newcommand{\bbC}{\mathbb C}
\newcommand{\bbG}{\mathbb G}
\newcommand{\bbF}{\mathbb F}
\newcommand{\bbH}{\mathbb H}
\newcommand{\bbP}{\mathbb P}
\newcommand{\bbN}{\mathbb N}
\newcommand{\bbQ}{\mathbb Q}
\newcommand{\bbR}{\mathbb R}
\newcommand{\bbV}{\mathbb V}
\newcommand{\bbZ}{\mathbb Z}
\newcommand{\adj}{\operatorname{adj}}
\newcommand{\Ad}{\mathrm{Ad}}
\newcommand{\Ann}{\mathrm{Ann}}
\newcommand{\rcris}{\mathrm{cris}}
\newcommand{\ch}{\mathrm{ch}}
\newcommand{\coker}{\mathrm{coker}}
\newcommand{\diag}{\mathrm{diag}}
\newcommand{\Diff}{\mathrm{Diff}}
\newcommand{\Dist}{\mathrm{Dist}}
\newcommand{\rDR}{\mathrm{DR}}
\newcommand{\ev}{\mathrm{ev}}
\newcommand{\Ext}{\mathrm{Ext}}
\newcommand{\cExt}{\mathcal{E}xt}
\newcommand{\fin}{\mathrm{fin}}
\newcommand{\Frac}{\mathrm{Frac}}
\newcommand{\GL}{\mathrm{GL}}
\newcommand{\Hom}{\mathrm{Hom}}
\newcommand{\hd}{\mathrm{hd}}
\newcommand{\rht}{\mathrm{ht}}
\newcommand{\id}{\mathrm{id}}
\newcommand{\im}{\mathrm{im}}
\newcommand{\inc}{\mathrm{inc}}
\newcommand{\ind}{\mathrm{ind}}
\newcommand{\coind}{\mathrm{coind}}
\newcommand{\Lie}{\mathrm{Lie}}
\newcommand{\Max}{\mathrm{Max}}
\newcommand{\mult}{\mathrm{mult}}
\newcommand{\op}{\mathrm{op}}
\newcommand{\ord}{\mathrm{ord}}
\newcommand{\pt}{\mathrm{pt}}
\newcommand{\qt}{\mathrm{qt}}
\newcommand{\rad}{\mathrm{rad}}
\newcommand{\res}{\mathrm{res}}
\newcommand{\rgt}{\mathrm{rgt}}
\newcommand{\rk}{\mathrm{rk}}
\newcommand{\SL}{\mathrm{SL}}
\newcommand{\soc}{\mathrm{soc}}
\newcommand{\Spec}{\mathrm{Spec}}
\newcommand{\St}{\mathrm{St}}
\newcommand{\supp}{\mathrm{supp}}
\newcommand{\Tor}{\mathrm{Tor}}
\newcommand{\Tr}{\mathrm{Tr}}
\newcommand{\wt}{\mathrm{wt}}
\newcommand{\Ab}{\mathbf{Ab}}
\newcommand{\Alg}{\mathbf{Alg}}
\newcommand{\Grp}{\mathbf{Grp}}
\newcommand{\Mod}{\mathbf{Mod}}
\newcommand{\Sch}{\mathbf{Sch}}\newcommand{\bfmod}{{\bf mod}}
\newcommand{\Qc}{\mathbf{Qc}}
\newcommand{\Rng}{\mathbf{Rng}}
\newcommand{\Top}{\mathbf{Top}}
\newcommand{\Var}{\mathbf{Var}}
\newcommand{\gromega}{\langle\omega\rangle}
\newcommand{\lbr}{\begin{bmatrix}}
\newcommand{\rbr}{\end{bmatrix}}
\newcommand{\forb}{\bigcirc\kern-2.8ex \because}
\newcommand{\forbb}{\bigcirc\kern-3.0ex \because}
\newcommand{\forbbb}{\bigcirc\kern-3.1ex \because}
\newcommand{\cd}{commutative diagram }
\newcommand{\SpS}{spectral sequence}
\newcommand\C{\mathbb C}
\newcommand\hh{{\hat{H}}}
\newcommand\eh{{\hat{E}}}
\newcommand\F{\mathbb F}
\newcommand\fh{{\hat{F}}}

\def\AA{{\mathcal A}}
\def\al{\alpha}
\def\bq{B_q(\ge)}
\def\bqm{B_q^-(\ge)}
\def\bqz{B_q^0(\ge)}
\def\bqp{B_q^+(\ge)}
\def\beneme{\begin{enumerate}}
\def\beq{\begin{equation}}
\def\beqn{\begin{eqnarray}}
\def\beqnn{\begin{eqnarray*}}
\def\bigsl{{\hbox{\fontD \char'54}}}
\def\bbra#1,#2,#3{\left\{\begin{array}{c}\hspace{-5pt}
#1;#2\\ \hspace{-5pt}#3\end{array}\hspace{-5pt}\right\}}
\def\cd{\cdots}
\def\CC{\hbox{\bf C}}
\def\ddd{\hbox{\germ D}}
\def\del{\delta}
\def\Del{\Delta}
\def\Delr{\Delta^{(r)}}
\def\Dell{\Delta^{(l)}}
\def\Delb{\Delta^{(b)}}
\def\Deli{\Delta^{(i)}}
\def\Delre{\Delta^{\rm re}}
\def\ei{e_i}
\def\eit{\tilde{e}_i}
\def\eneme{\end{enumerate}}
\def\ep{\epsilon}
\def\eeq{\end{equation}}
\def\eeqn{\end{eqnarray}}
\def\eeqnn{\end{eqnarray*}}
\def\fit{\tilde{f}_i}
\def\FF{{\rm F}}
\def\ft{\tilde{f}}
\def\gau#1,#2{\left[\begin{array}{c}\hspace{-5pt}#1\\
\hspace{-5pt}#2\end{array}\hspace{-5pt}\right]}
\def\ge{\hbox{\germ g}}
\def\gl{\hbox{\germ gl}}
\def\hom{{\hbox{Hom}}}
\def\ify{\infty}
\def\io{\iota}
\def\kp{k^{(+)}}
\def\km{k^{(-)}}
\def\llra{\relbar\joinrel\relbar\joinrel\relbar\joinrel\rightarrow}
\def\lan{\langle}
\def\lar{\longrightarrow}
\def\max{{\rm max}}
\def\lm{\lambda}
\def\Lm{\Lambda}
\def\mapright#1{\smash{\mathop{\longrightarrow}\limits^{#1}}}
\def\mm{{\bf{\rm m}}}
\def\nd{\noindent}
\def\nn{\nonumber}
\def\nnn{\hbox{\germ n}}
\def\catob{{\mathcal O}(B)}
\def\oint{{\mathcal O}_{\rm int}(\ge)}
\def\ot{\otimes}
\def\op{\oplus}
\def\opi{\ovl\pi_{\lm}}
\def\ovl{\overline}
\def\plm{\Psi^{(\lm)}_{\io}}
\def\qq{\qquad}
\def\q{\quad}
\def\qed{\hfill\framebox[2mm]{}}
\def\QQ{\hbox{\bf Q}}
\def\qi{q_i}
\def\qii{q_i^{-1}}
\def\ra{\rightarrow}
\def\ran{\rangle}
\def\rlm{r_{\lm}}
\def\ssl{\mathfrak{sl}}
\def\slh{\widehat{\ssl_2}}
\def\ge{\hbox{\germ g}}
\def\ti{t_i}
\def\tii{t_i^{-1}}
\def\til{\tilde}
\def\tm{\times}
\def\tt{{\hbox{\germ{t}}}}
\def\ttt{\hbox{\germ t}}
\def\ua{U_{\AA}}
\def\ue{U_{\vep}}
\def\uq{U_q(\ge)}
\def\ufin{U^{\rm fin}_{\vep}}
\def\ufinp{(U^{\rm fin}_{\vep})^+}
\def\ufinm{(U^{\rm fin}_{\vep})^-}
\def\ufinz{(U^{\rm fin}_{\vep})^0}
\def\uqm{U^-_q(\ge)}
\def\uqp{U^+_q(\ge)}
\def\uqmq{{U^-_q(\ge)}_{\bf Q}}
\def\uqpm{U^{\pm}_q(\ge)}
\def\uqq{U_{\bf Q}^-(\ge)}
\def\uqz{U^-_{\bf Z}(\ge)}
\def\ures{U^{\rm res}_{\AA}}
\def\urese{U^{\rm res}_{\vep}}
\def\uresez{U^{\rm res}_{\vep,\ZZ}}
\def\util{\widetilde\uq}
\def\uup{U^{\geq}}
\def\ulow{U^{\leq}}
\def\bup{B^{\geq}}
\def\blow{\ovl B^{\leq}}
\def\vep{\varepsilon}
\def\vp{\varphi}
\def\vpi{\varphi^{-1}}
\def\VV{{\mathcal V}}
\def\xii{\xi^{(i)}}
\def\Xiioi{\Xi_{\io}^{(i)}}
\def\WW{{\mathcal W}}
\def\wtil{\widetilde}
\def\what{\widehat}
\def\wpi{\widehat\pi_{\lm}}
\def\ZZ{\mathbb Z}
\def\spsp(#1,#2){\begin{pmatrix}
#1,\\ \hline#2\end{pmatrix}}

\theoremstyle{definition}
\newtheorem{df}{Definition}[section]
\newtheorem{pro}[df]{Proposition}
\newtheorem{thm}[df]{Theorem}
\newtheorem{lem}[df]{Lemma}
\newtheorem{ex}[df]{Example}
\newtheorem{cor}[df]{Corollary}
\newtheorem{con}[df]{Conjecture}

\renewcommand{\thesection}{\arabic{section}}
\section{Introduction}
\setcounter{equation}{0}
\renewcommand{\theequation}{\thesection.\arabic{equation}}

The notion of pictures is initiated by James and Peel \cite{JP} and 
Zelevinsky \cite{Z}, which is roughly a bijective map
between two skew Young diagrams with certain conditions (See Sect.2).
Let $\lm,\mu,\nu$ be Young diagrams with $|\mu|=|\nu\setminus\lm|$ and 
denote the set of pictures
from $\mu$ to $\nu\setminus\lm$ by ${\bf P}(\mu,\nu\setminus\lm)$. 
Then one has the following remarkable result:
\begin{equation}
\sharp\Pic=c_{\lm,\mu}^{\nu},\label{LR}
\end{equation}
where $c_{\lm,\mu}^{\nu},$ is the usual Littlewood-Richardson number,
which is shown in \cite{F}.

The theory of crystal bases is introduced by
Kashiwara (\cite{K1},\cite{K2}), which is widely applied
to many areas in mathematics and physics, in particular,
combinatorial representation theory. In \cite{KN}, it is revealed that
crystal bases for classical Lie algebras are presented by
'Young tableaux' and in \cite{N1} by the first author
it is shown that so-called Littlewood Richardson rule for tensor 
products of representations are described by crystal bases
(see Sect.3).
So, together with (\ref{LR}) we deduced certain one to one
correspondence between pictures and crystal bases, which is given in
Theorem \ref{main}.

This article is organized as follows. In Sect.2, we introduce pictures.
In Sect.3, we review the crystal bases of type $A_n$ and the
description of
Littlewood-Richardson rule in terms of crystal bases.
In Sect.4, we shall state the main theorem, namely, we shall give an
explicit one to one correspondence between pictures and 
Littlewood-Richardson crystals
of type $A_n$. In the subsequent three sections,
we shall give a proof of the theorem. In the last section, 
we shall generalize
the notion of pictures and give certain conjecture on it.

The authors would like to 
thank M.Kashiwara and M.Okado for their comments
and advices on this work.

\renewcommand{\thesection}{\arabic{section}}
\section{Young Tableaux and Pictures}
\setcounter{equation}{0}
\renewcommand{\theequation}{\thesection.\arabic{equation}}

\subsection{Young Tableaux}
Let $\lambda=(\lambda_1,\lambda_2,\cdots,\lambda_m)$
be a Young diagram or a partition, which satisfies
$\lambda_1 \geq\lambda_2 \geq \cdots \geq \lambda_m \geq 0$.
We usually write a Young diagram by using square boxes:
\begin{ex} For $\lambda = (2,2,1)$, write
$\lambda = \twotwoone(, , , , )$
\end{ex}
In this article we frequently use the following
coordinated expression 
for a Young diagram, that is, we identify a Young diagram 
with a subset of $\bbN\times\bbN$:

\quad\quad$\hspace{-3mm}{\color{white}\hfourbox({\color{black}1},{\color{black}2},{\color{black}3},{\color{black}4})}$

${\color{white}\vfourbox({\color{black}1},{\color{black}2},{\color{black}3},{\color{black}4})}$
$\begin{array}{l}\hspace{-4mm}\fourthreetwo(,,,,,,a,,)\,\,\,\\ \q\end{array}
$.\quad
In this diagram, the coordinate of $a$ is $(2,3)$.

\begin{ex}@For a Young diagram $\lambda = \twotwoone(,,,,)$, 
its coordinated expression is 
$\lambda = \left\{
(1,1), (1,2),
(2,1), (2,2),
(3,1)
\right\}$.
\end{ex}

\begin{df}
A numbering of a 
Young diagram $\lambda$ is called a {\it Young tableau} 
of shape $\lm$ if it satisfies
\begin{enumerate}
\item
In each row, all entries weakly increase from left to right.
\item
In each column, all entries increase from top to bottom.
\end{enumerate}
\end{df}
Note that it is also called 'semi-standard tableau'.
In this article, we prefer Young tableau to semi-standard 
tableau following \cite{F}.

For a Young tableau $T$ of shae $\lm$, 
we also consider a coordinate 
like as $\lm$. 
Then an entry of $T$ in $(i,j)$ is denoted by $T_{i,j}$ and 
called $(i,j)$-entry.
For $k>0$, define 
\begin{equation}
\label{Y-diag}
T^{(k)} = \{ (l,m) \in \lm | T_{l,m} = k \}.
\end{equation}
{\sl Remark.}
Note that in $T^{(k)}$, there is no two elements in one column.
Thus, we can write 
\[
 T^{(k)}=\{(a_1,b_1),~(a_2,b_2),\cd,(a_m,b_m)\},
\]
with 
$a_1\leq a_2\leq \cd a_m$ and $b_1>b_2>\cd >b_m$.
If $(i,j)$-entry in a tableau $T$ is $k$ and 
$(i,j)=(a_p,b_p)$ in $T^{(k)}$ as above, we define a function $p(T;i,j)$ by 
\begin{equation}
 p(T;i,j)=p,
\label{pT}
\end{equation}
that is, $p(T;i,j)$ is the number of $(i,j)$-entry 
from the right in $T^{(k)}$.
It is immediate from the definition:
\begin{equation}
\text{If }T_{i,j}=T_{x,y}\text{ and }p(T;i,j)=p(T;x,y),
\text{ then }(i,j)=(x,y)
\label{ijxy}
\end{equation}
\begin{ex}
For 
$T=\twotwoone(1,2,2,3,4)$, we have 
$T_{1,1} = 1,T_{1,2}=2, T_{2,1} = 2, \\
T_{2,2}=3, T_{3,1} =4$, $p(T;1,2)=1$, $p(T;2,1)=2$
\end{ex}

\begin{df}
Let $\lambda$ and $\mu$ be Young diagrams with $\mu\subset\lm$.
A {\it skew diagram} $\lambda \setminus \mu$ is obtained by 
removing $\mu$ from $\lm$.
\end{df}
\begin{ex}
For $\lambda = (2,2), \mu = (1)$, we have 
$\lambda \setminus \mu =\normalbaselines\m@th\offinterlineskip\vcenter{\hbox{{\color{white}\fsquare(0.3cm,)}\fsquare(0.3cm,)}\vskip-0.4pt\hbox{\fsquare(0.3cm,)\addsquare(0.3cm,)}}$
\end{ex}
\subsection{Picture}

Now, let us introduce the notion of pictures.
\begin{df}
(Orders $<_P$ and $<_J$)
We define the following two kinds of orders on
a subset $X\subset\bbN \times \bbN$: For
$(a,b), \,\,(c,d)\in X$,
\begin{item}
(1) $\leq_P$: \quad $(a,b)\leq_P(c,d)$ iff
$a\leq c \text{ and } b \leq d$.\\
(2) $ \leq_J$: \quad $(a,b) \leq_J (c,d)$ iff
$a < c \text{, or }  a = c \text{ and } b\geq d$.\\
\end{item} 
\end{df}
Note that the order $\leq_P$ is a partial order 
and $\leq_J$ is a total order.

\begin{df}[\cite{Z}] 
Let $X,Y \subset \bbN \times \bbN$.
\begin{item}
(1) 
A map $f : X \to Y$ is said to be 
{\it PJ-standard} if it satisfies
\[
\text{For }(a,b),(c,d) \in X, 
\text{if }(a,b) \leq_P (c,d), \text{ then }  
f(a,b) \leq_J f(c,d).
\]
(2)
A map $f:X\to Y$ is a {\it picture} if it is 
bijective and
both $f$ and $f^{-1}$ are PJ-standard.
\end{item}
\end{df}

Taking three Young diagrams
$\lambda, \mu, \nu\subset \bbN\times\bbN$, 
denote the set of pictures by:
\[
\Pic := \{ f : \mu \to \nu \setminus \lambda 
\,|\,f\text{ is a picture.}\}
\]

\renewcommand{\thesection}{\arabic{section}}
\section{Crystal Bases and Young tableaux}
\setcounter{equation}{0}
\renewcommand{\theequation}{\thesection.\arabic{equation}}

Crystal bases of type $A_n$ is realized in terms of Young tableaux
(\cite{KN}).

Let $\Lm_i$ (resp. $\al_i$, $h_i$)($i=1,2,\cd,n$) 
be the fundamental weight (resp. simple root, simple coroot) of type
$A_n$.

Let $B_1:=\{\fsquare(0.3cm,i)|i=1,2\cd,n+1\}$ be the crystal of type 
$A_n$ for the fundamental weight $\Lm_1$.
A dominant weight $\lm$ is identified with a Young diagram in usual
way. Then, we use the same notation for a dominant weight and the
corresponding Young diagram.
Let $\lm$ be a Young diagram with a depth at most $n$ and $|\lm|=N$.
Then the crystal $B(\lm)$ is embedded in $B_1^{\otimes N}$ 
and realized by Young tableaux (\cite{KN}) . This embedding, say 
{\it reading}, 
is not unique. Now, we introduce 
two of them. One is the {\it middle-eastern reading} and the other is 
the {\it far-eastern reading} (\cite{HK}).

\begin{df}
Let $T$ be a Young tableau of shape $\lm$ with entries $\{1,2,\cd, n+1\}$.
\begin{enumerate}
\item
We read the entries in $T$ each row from right to left and from the top
row to the bottom row. Then  the resulting sequence of the entries 
$i_1,i_2,\cd,i_N$ gives the embedding of crystals:
\[
 B(\lm)\hookrightarrow B_1^{\ot N} \q
(T\mapsto \fsq(i_1)\ot\cd\ot\fsq(i_N)),
\]
which is called middle-eastern reading and denoted by $\rm ME$.
\item
We read the entries in $T$ each column from the top to the 
bottom and from the right-most column to the left-most column.
Then  the resulting sequence of the entries 
$i_1,i_2,\cd,i_N$ gives the embedding of crystals:
\[
 B(\lm)\hookrightarrow B_1^{\ot N} \q
(T\mapsto \fsq(i_1)\ot\cd\ot\fsq(i_N)),
\]
which is called far-eastern reading and denoted by $\rm FE$.
\end{enumerate}
\end{df}

\begin{ex}@\quad For a Young tableau
$T = \fourthreeone(1,2,2,3,2,3,4,5)$, we have
\begin{eqnarray*}
&&
{\rm ME}(T) = \fsq(3) \otimes \fsq(2) \otimes 
\fsq(2) \otimes \fsq(1) \otimes \fsq(4) 
\otimes \fsq(3) \otimes \fsq(2) \otimes \fsq(5),\\
&&
{\rm FE}(T)= \fsq(3) \otimes \fsq(2) \otimes \fsq(4) 
\otimes \fsq(2) \otimes \fsq(3) \otimes \fsq(1) 
\otimes \fsq(2) \otimes \fsq(5).
\end{eqnarray*}
\end{ex}

\begin{df}(Addition)
For $i\in\{1,2,\cd,n+1\}$ and a Young diagram 
$\lambda=(\lambda_1,\lambda_2,\cdots, \lambda_n)$, we define
\[
\lambda[i]:=(\lambda_1,\lambda_2,\cdots,\lambda_i+1,\cdots, \lambda_n)
\]
which is said to be an {\it addition} of $i$ to $\lm$.
In general, for $i_1,i_2,\cd,i_N\in\{1,2,\cd,n+1\} $ and a Young diagram 
$\lm$, we define 
\[
 \lambda[i_1,i_2,\cdots,\i_N]:=(\cd ((\lambda[i_1])[i_2])\cd)[i_N],
\]
which is called an {\it addition} of $i_1,\cd,i_N$ to $\lm$.
\end{df}

\begin{ex}
For a sequence ${\bf i}=31212$, 
the addition of $\bf i$ to $\lambda = \twoone(,,)$ is:
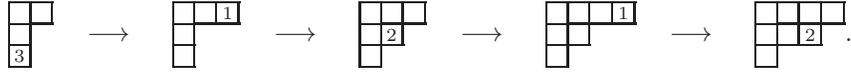
\[
\twooneone(,,,3) \quad \longrightarrow \quad \threeoneone(,,1,,)
\quad \longrightarrow \quad \threetwoone(,,,,2,)
\quad \longrightarrow \quad \fourtwoone(,,,1,,,) \quad \longrightarrow
 \quad \fourthreeone(,,,,,,2,).
\]
\end{ex}
\vspace{0.5cm}
{\sl Remark.}
For a Young diagram $\lm$, an addition $\lm[i_1,\cd,i_N]$ is not 
necessarily a Young diagram.

\begin{ex}
For a sequence ${\bf i'}=22133$, 
the addition of $\bf i'$ to 
$\lambda = \twoone(,,)$ is 
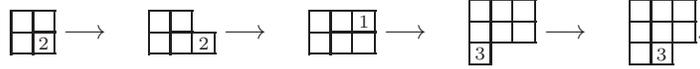
\[
\twotwo(,,,2) \longrightarrow \quad \twothree(,,,,2) \longrightarrow \quad
\threethree(,,1,,,) \longrightarrow \quad  \threethreeone(,,,,,,3)
 \longrightarrow \quad \threethreetwo(,,,,,,,3).
\]
Then we see that $\lm[2,2]$ is not a Young digram.
\end{ex}

\subsection{Littlewood-Richardson rule}
As an application of the description of crystal bases of type $A_n$,
we see so-called ``Littlwood-Richardson rule'' of type $A_n$.

For a sequence $i_1,i_2,\cd,i_N\in \{1,2,\cd,n+1\}$ and 
a Young diagram $\lm$, let $\til\lm:=\lm[i_1,i_2,\cd,i_{N}]$ be
an addition of $i_1,i_2,\cd,i_N$ to $\lm$. Then set
\[
{\mathbf B}(\til\lm)=\begin{cases}
{\mathbf B}(\til\lm)&\text{ if }\lm[i_1,\cd,i_k]
\text{ is a Young diagram for any }k=1,2,\cd,N,\\
\emptyset&\text{otherwise.}
\end{cases}
\]

\begin{thm}[\cite{N1}]
\label{LRrule}
Let $\lambda$ and $\mu$ be Young diagrams with at most $n$ rows.
Then we have
\begin{equation}
\mathbf{B}(\lambda) \otimes \mathbf{B}(\mu) \cong  
\bigoplus_{\tiny\begin{array}{l}T\in \mathbf{B}(\mu),\\
{\rm FE}(T)= \fsq(i_1) 
\otimes\cdots\otimes\fsq(i_N)
\end{array}} 
\mathbf{B}(\lambda[i_1,i_2,\ldots,i_N]).
\label{LR}
\end{equation}
\end{thm}
Note that this also holds for $\rm ME$.

Let $c_{\lm,\mu}^{\nu}$ be the multiplicity of $\mathbf{B}(\nu)$ in 
$\mathbf{B}(\lm)\ot\mathbf{B}(\mu)$, which is denoted by 
$c_{\lm,\mu}^{\nu}$ and called the 
Littlewood-Richardson number. We have the following:
\begin{thm}[\cite{F}]
$\sharp\Pic=c_{\lm,\mu}^{\nu}$.
\end{thm}
For Young diagrams 
$\lambda, \mu, \nu$, we define
\[
\Cry := 
\left\{T\in{\mathbf B}(\mu)|
\begin{array}{l} 
\text{ME}(T)=\MECry, \\
\text{for any }k=1,\cd,N,\\
\lm[i_1,\cd,i_k]\text{ is a Young diagram and}\\
\lm[i_1,\cd,i_N]=\nu.
\end{array}
\right\},
\]
whose element is called a {\it Littlewood-Richardson crystal} with
respect to a triplet 
$(\lm,\mu,\nu)$.
Then by Theorem \ref{LRrule}, we have
\begin{cor}
$\sharp\Pic=\sharp\Cry$.
\end{cor}
We shall see an explicit one-to-one correspondence between 
$\Pic$ and $\Cry$ in the next section.
\renewcommand{\thesection}{\arabic{section}}
\section{Main Theorem}
\setcounter{equation}{0}
\renewcommand{\theequation}{\thesection.\arabic{equation}}

For Young diagrams $\lambda, \mu, \nu$, we have two sets:
$\Pic $ and 
$\Cry$.
In case 
$|\lambda|+|\mu|=|\nu|$, we define the following 
map $\Phi:\Pic\to\Cry$:
For $f=(f_1, f_2)\in\Pic$, set 
\[
\Phi(f)_{i,j}:= f_{1}(i,j),
\]
that is, $\Phi(f)$ is a filling of shape $\mu$ and 
its $(i, j)$-entry is given as $f_1(i, j)$.

Furthermore, for a crystal $T \in \Cry$, 
define a  map 
$\Psi : \Cry \to \Pic$ by 
\[
\Psi(T) : (i,j)\in\mu \mapsto (T_{i,j}, 
\lambda_{T_{i,j}}+p(T;i,j))\in\nu\setminus\lambda,
\]
where 
$p(T;i,j)$ as in (\ref{pT}).

The following is the main theorem in this article.
\begin{thm}\label{main}
For Young diagrams $\lm,\mu,\nu$ as above,
the map $\Phi:\Pic\to\Cry$ is a bijection and 
the map $\Psi$ is the inverse of $\Phi$.
\end{thm}

\begin{ex}\label{ff}
Take $\lambda=(3,1,1)=\threeoneone(,,,,),$ 
$\mu=(3,2)=\threetwo(,,,,)$ and 
$\nu=(4,3,2,1)=\fourthreetwoone$. 
As subsets in $\bbN\times\bbN$, we have\\
$\mu = \{ (1,1), (1,2), (1,3), (2,1), (2,2) \}$,
$\nu \setminus \lambda=
\{(1,4), (2,2), (2,3), (3,2),(4,1) \}$.

In this case $\sharp\Pic=2$. Set $\Pic=
\{f, f'\}$ and their
explicit forms are 

$f=$
\begin{tabular}{c||c|c|c|c|c}
$\mu$ & $(1,1)$ & $(1,2)$ & $(1,3)$ & $(2,1)$ & $(2,2)$\\
\hline
$\nu \setminus \lambda$ & $(1,4)$ & $(2,3)$ & $(2,2)$ & $(3,2)$ & (4,1)\\
\end{tabular}

\vspace{0.3cm}

$f' =$
\begin{tabular}{c||c|c|c|c|c}
$\mu$ & $(1,1)$ & $(1,2)$ & $(1,3)$ & $(2,1)$ & $(2,2)$\\
\hline
$\nu \setminus \lambda$ & $(1,4)$ & $(2,2)$ & $(4,1)$ & $(2,3)$ & (3,2)\\
\end{tabular}

We have
\[
\Cry = \{T= \threetwo(1,2,2,3,4),\,\,
T'= \threetwo(1,2,4,2,3)\},
\]
and $\Phi(f) = T$, $\Phi(f') = T'$.
\end{ex}

In the subsequent sections, 
let us give the proof of 
Theorem \ref{main},
which consists in the following three steps:
\begin{enumerate}
\item
Well-definedness of the map $\Phi$.
\item
Well-definedness of the map $\Psi$.
\item
Bijectivity of $\Phi$ and $\Psi=\Phi^{-1}$.
\end{enumerate}

\renewcommand{\thesection}{\arabic{section}}
\section{Well-definedness of $\Phi$}
\setcounter{equation}{0}
\renewcommand{\theequation}{\thesection.\arabic{equation}}

For the well-definedness of $\Phi$, 
it suffices to show:
\begin{pro}\label{prop:phi}
Let $\lm,\mu$ and $\nu$ be Young diagrams with
$|\lm|+|\mu|=|\nu|$. 
\begin{enumerate}
\item
For any $f\in\Pic$, $\Phi(f)$ is a Young tableau of shape
$\mu$, that is, $\Phi(f)\in B(\mu)$.
\item
Writing ${\rm ME}(\Phi(f)) = \MECry$, for any $k=1,\cd,N$,
$\lambda[i_1, i_2, \cdots ,i_k]$ is a Young diagram and 
$\lm[i_1,\cd,i_N]=\nu$.
\end{enumerate}
\end{pro}

\subsection{Proof of Proposition~\ref{prop:phi}(1)}
For $f\in\Pic$, it is immediate from the definition of $\Phi$ that the
shape of $\Phi(f)$ is $\mu$.
Next, in order to see 
that $\Phi(f)$ is a Young tableau, we may show 

(a) $\Phi(f)_{i,j}\leq \Phi(f)_{i,j+1}$.\qq
(b) $\Phi(f)_{i,j}<\Phi(f)_{i+1,j}$.

By the definition of $\Phi$, one has
\[
\Phi(f)_{i,j} = f_1(i,j),\quad
\Phi(f)_{i,j+1} = f_1(i,j+1)
\]
Since $(i,j) <_P (i,j+1)$ and $f$ is a picture, 
\[
(f_1(i,j), f_2(i,j)) <_J (f_1(i,j+1), f_2(i,j+1))
\]
Then, by the definition of $<_J$ one gets
\[
\Phi(f)_{i,j} = f_1(i,j) \leq f_1(i,j+1) = \Phi(f)_{i,j+1}
\]
which shows (a).

By the definition of $\Phi$ again, one has
\[
\Phi(f)_{i,j} = f_1(i,j)\text{ and }
\Phi(f)_{i+1,j} = f_1(i+1,j).
\]
Since $(i,j) <_P (i+1,j)$ and $f$ is a picture, 
\[
(f_1(i,j), f_2(i,j)) <_J (f_1(i+1,j), f_2(i+1,j)),
\]
which implies $f_1(i,j) \leq f_1(i+1,j)$.
Here, suppose that $f_1(i,j) = f_1(i+1,j)$. 
It follows from the definition of $<_J$ that 
\[
f_2(i,j) > f_2(i+1,j).
\]
This means
\[
(f_1(i,j), f_2(i,j)) _P> (f_1(i+1,j), f_2(i+1,j)).
\]
Since $f$ is a picture, applying $f^{-1}$ to this one has 
\[
(i,j) _J> (i+1,j), 
\]
which derives a contradiction. Thus, 
one gets $f_1(i,j) < f_1(i+1,j)$, that is, 
$\Phi(f)_{i,j} < \Phi(f)_{i+1,j}$, or 
equivalently, (b).
Now, we obtain $\Phi(f)\in B(\mu)$.

\subsection{Addition and Picture}\label{add:pic}

Before showing 
Proposition~\ref{prop:phi}(2), we prepare the 
lemma as below:
\begin{lem}~\label{lem:add}
Let $f:\mu\to\nu\setminus\lambda$ be a picture and 
set 
${\rm ME}(\Phi(f)) = \MECry$.
Let $(p_k,q_k)\in\mu$ be the place of 
$\fsquare(0.4cm,i_k)$ in $\Phi(f)\in B(\mu)$ and 
$(a_k, b_k) \in \nu$ the place of the $k$-th 
addition in $\lm[i_1,\cd,i_N]$. Then we have 
$f(p_k,q_k) = (a_k, b_k)$ for any $k=1,\cd,N$.
\end{lem}

\begin{ex}
For a picture $f = $
\begin{tabular}{c||c|c|c|c|c}
$\mu$ & $(1,1)$ & $(1,2)$ & $(1,3)$ & $(2,1)$ & $(2,2)$\\
\hline
$\nu \setminus \lambda$ & $(1,4)$ & $(2,3)$ & $(2,2)$ & $(3,2)$ & $(4,1)$\\
\end{tabular}

as in Example \ref{ff}, we have $\Phi(f)=\threetwo(1,2,2,3,3)$ and 
${\rm ME}(\Phi(f)) = \fsquare(0.4cm,2) \otimes 
\fsquare(0.4cm,2) \otimes \fsquare(0.4cm,1) 
\otimes\fsquare(0.4cm,3) \otimes \fsquare(0.4cm,3)$.

Now, let us see the second $\fsquare(0.3cm,i_2)
= \fsquare(0.3cm,2)$. This is added to 
the second row of $\lm$ by the addition:
$\fourthreetwo(,,,,,,2,,)$, and then 
it is placed in $(2,3)\in\nu$. 

The place of $\fsquare(0.3cm,i_2)
=\fsquare(0.3cm,2)$ in $\mu$ is $(1,2)$ and 
$f(1,2)=(2,3)$.
\end{ex}

\begin{proof}
Set $m:=i_k$. List the $m$-th row in 
$\nu \backslash \lambda$ according to the order $<_P$:
\[
(m, \lambda_m+1) <_P (m, \lambda_m+2) <_P \cdots <_P 
(m, \lambda+c_m) = (m,\nu_m).
\]
Since $f$ is a picture, one has 
\[
f^{-1}(m, \lambda_m+1) <_J f^{-1}(m, \lambda_m+2) 
<_J \cdots <_J f^{-1}(m, \nu_m).
\]
Since the middle-eastern reading follows the order 
$<_J$, 
$(p,q):=f^{-1}(m, \lambda_m+j)$ 
$(j=1,\cd,\nu_m-\lm_m)$
is added $j$-th to the $m$-th row of $\lm$, which implies that 
the entry in $(p,q)\in\mu$ is added to 
$(m, \lambda_m+j)\in\nu$.

On the other-hand  $f(p,q)=
f(f^{-1}(m, \lambda_m+j))=(m, \lambda_m+j)$, 
which completed the proof of the lemma.
\end{proof}

\subsection{Proof of Proposition~\ref{prop:phi} (2)}
Due to the definition of $\Phi$, it is easy to see that 
the number of entry $i$ $(i=1,\cd, n+1)$ is equal to 
$\nu_i-\lm_i$, which implies $\lm[i_1,\cd,i_N]=\nu$.

Writing ${\rm ME}(\Phi(f)) = \MECry$, 
let us show that $\lm[i_1,\cd,i_k]$ is a 
Young diagram for any $k$ by the induction on $k$.

In case $k = 1$. Denote $\Phi(f)$ by $T$.
Let us show that $\lambda[i_1] 
= \lambda[T_{1,\mu_1}]$ is a Young diagram.
Since $(1,\mu_1)$ is the minimum element in 
$\mu$ with respect to the order $<_J$ and 
$f$ is a picture, 
$f(1,\mu_1) = (f_1(1,\mu_1), f_2(1,\mu_1))$ must 
be minimal with respect to the order $<_P$.
Set $s:=f_1(1, \mu_1)$. Then, by Lemma \ref{lem:add} 
we have 
$f_2(1,\mu_1) = \lambda_s + 1$. 
Assume that there is no box above 
$(s, \lambda_s + 1)$ in $\lm[i_1]$.
$\nu=\lm[i_1,\cd,i_N]$ is a Young diagram, 
which means that there is some $j$ such that 
$i_j$ is added above $(s, \lambda_s + 1)$.
Since $f(1,\mu_1)$ is minimal in 
$\nu\backslash \lambda$ with respect to $<_P$, in 
the addition of ${\rm ME}(\Phi(f))$ to $\lm$, 
nothing is added 
above $(s, \lambda_s + 1)$ after $i_1$,
which derives a contradiction. Then, we know that 
there is originally a box above $(s, \lambda_s + 1)$ and
then shows 
\[
\lambda_{s-1} - \lambda_s > 0.
\]
Therefore, $\lambda[i_1]$ is a Young diagram.

In case $k = m>1$. 
Suppose $\lm':=\lambda[i_1,i_2, \cdots, i_{m-1}]$
to be a Young diagram and set $i_m:= T_{x,y}$, namely, 
$i_m$ is the $(x,y)$-entry in $T$.
By considering similarly to the case $k=1$, 
$f(x,y)$ must be  minimal in 
$\nu\backslash \lambda'$ with respect to the 
order $<_P$.
By Lemma \ref{lem:add}, the destination 
of $i_m$ by the addition is 
$f(x,y) = (i_m, \lambda'_{i_m} + 1)$.
Then, nothing comes above $f(x,y)$ after $i_m$. 
Thus, by arguing similarly to the case $k=1$, 
we have 
\[
\lambda'_{i_m-1}  - \lambda'_{i_m} > 0,
\]
and then $\lambda'[i_m]$ is a Young diagram.
\qed

\renewcommand{\thesection}{\arabic{section}}
\section{Well-definedness of $\Psi$}
\setcounter{equation}{0}
\renewcommand{\theequation}{\thesection.\arabic{equation}}

In this section, we shall show 
the well-definedness of $\Psi$, 
that is, the image $\Psi(\Cry)$ is in $\Pic$.
Let $\lm,\mu,\nu$ be as above.
\begin{pro}\label{prop:psi}
For $T\in\Cry$, we have 
\begin{enumerate}
\item
$\Psi(T)$ is a map from $\mu$ to 
$\nu\backslash \lambda$ and $\Psi(T)(\mu)=\nu\setminus\lm$.
\item
$\Psi(T)$ is a bijection.
\item
Both $\Psi(T)$ and $\Psi(T)^{-1}$ are 
PJ-standard.
\end{enumerate}
\end{pro}
Before starting the proof, we prepare one lemma:
\begin{lem}\label{lem:tabad}
For $T \in \Cry$ and $(i,j)\in\mu$, 
define $(p,q):=\Psi(T)(i,j)$. 
Then we have 
that the destination of $(i,j)$ by 
the addition of ${\rm ME}(T)$ is equal to $(p,q)$.
\end{lem}

\begin{proof}
Set $m:=T_{i,j}$ and let $(i,j)$ be the 
$p$-th element in $T^{(m)}$ from the right, where 
$T^{(m)}$ is as in (\ref{Y-diag}).
Then, by the addition, 
$T_{i,j}$ is added $p$-th to the $m$-th row in $\nu$.
By the definition of $\Psi$, one has 
$\Psi(T)(i,j) = (m,\lambda_m+p)$. This shows the lemma.
\end{proof}

\subsection{Proof of Proposition~\ref{prop:psi} (1)}

It is clear from the definition of $\Psi$ that 
$\Psi(T)$ is a map from $\mu$.
Since $T \in \Cry$, 
one has that for any $j=1\cd,n$ 
the number of $j$ 
in $T$ is equal to $\nu_j-\lm_j$. Then it follows from Lemma 
\ref{lem:tabad} that 
$\Psi(T)(\mu)=\nu\backslash \lambda$.
Thus, we have (1).

\subsection{ Proof of proposition~\ref{prop:psi} (2)}

Since $|\mu| = |\nu\setminus\lambda|$ and $\Psi(T)=\nu\setminus\lm$ by
Proposition~\ref{prop:psi} (1), it suffices to show that 
$f:=\Psi(T)$ is injective.
By the definition of $\Psi$, for $(i,j), (x,y)\in \mu$ there are some 
$p$ and $q$ such that 
\[
f(i,j) = (T_{i,j}, \lambda_{T_{i,j}}+p),\qq
f(x,y) = (T_{x,y}, \lambda_{T_{x,y}}+q).
\]
Indeed, $p=p(T;,i,j)$ and $q=p(T;x,y)$.
Suppose that $f(i,j) = f(x,y)$. One has
\[
T_{i,j} = T_{x,y},\quad
\lambda_{T_{i,j}}+p = \lambda_{T_{x,y}}+q.
\]
Then $p = q$. Hence, by (\ref{ijxy}) 
one has $(i,j) = (x,y)$ and then $f$ is injective.

\subsection{Proof of Proposition~\ref{prop:psi} (3)}

First, let us see $f=\Psi(T)$ to be PJ-standard.
For the purpose, we may show for any $(i,j)\in\mu$,

(a) $f(i,j)<_J f(i,j+1)$. \q (b) $f(i,j)<_J f(i+1,j)$. 

\noindent
(a) For $(i,j), (i,j+1)\in\mu$, 
there are some $p$ and $q$ such that
\[
f(i,j) = (T_{i,j}, \lambda_{T_{i,j}}+p), \qq
f(i,j+1) = (T_{i,j+1}, \lambda_{T_{i,j+1}}+q).
\]
Since $T$ is a Young tableau, one has
\[
T_{i,j} \leq T_{i,j+1}.
\]
If $T_{i,j} < T_{i,j+1}$, this implies $f(i,j) <_J f(i,j+1)$ and then 
there is nothing to show.
So, assume $T_{i,j} = T_{i,j+1}=:m$. 
In this case, $(i,j), (i,j+1)\in T^{(m)}$ and they are neighboring each other.
Thus, we have $p=q+1$ and then 
\[
\lambda_{T_{i,j}}+p > \lambda_{T_{i,j+1}}+q.
\]
This shows
$f(i,j) <_J f(i,j+1)$.

\noindent
(b)
For $(i,j), (i+1,j)\in \mu$, there are some $p$ and $r$ such that 
\[
f(i, j) = (T_{i,j}, \lambda_{T_{i,j}}+p), \qq
f(i+1, j) = (T_{i+1,j}, \lambda_{T_{i+1,j}}+r).
\]
Since $T$ is a Young tableau, we have 
\[
T_{i,j} < T_{i+1,j},
\]
which means $f(i,j) <_J f(i+1,j)$
and then $f$ is PJ-standard.
\vspace{0.5cm}

Next, let us show $f^{-1}$ to be PJ-standard.
It is sufficient to see that for $(a, b), (a, b+1), (a+1,b)\in 
\nu \backslash \lambda$:
\[
\text{(c)}\,\, f^{-1}(a,b)<_J  f^{-1}(a,b+1).\qq  
\text{(d)}\,\,f^{-1}(a,b)<_J  f^{-1}(a+1,b).
\]
Set
\[
(i,j):= f^{-1}(a,b),\quad (x,y):= f^{-1}(a,b+1), \quad (s,t):=
f^{-1}(a+1,b).
\]
(c) There exist $p$ and $q$ such that 
\[
(a, b) = f(i,j) = (T_{i,j}, \lambda_{T_{i,j}}+p),\qq
(a,b+1) = f(x,y) = (T_{x,y}, \lambda_{T_{x,y}}+q).
\]
Thus, we have 
\[
T_{i,j} = T_{x,y} = a,\qq
\lambda_a+p = b,\qq
\lambda_a+q = b+1,
\]
which implies $q=p+1$. Then we know that 
$(i,j)$ and $(x,y)$ are neighboring in $T^{(a)}$ and then
$i = x$ and $j > y$, or $i < x$. 
Therefore, 
\[
f^{-1}(a,b) = (i,j) <_J (x,y) = f^{-1}(a,b+1),
\]
and then we show (c).

\noindent
(d) 
There is $(a,b)$ just above $(a+1,b)$ in the same column in 
$\nu\backslash \lambda$.
It follows from Lemma \ref{lem:tabad} that in the addition of ${\rm ME}(T)$, 
$T_{i,j}$ is added earlier than $T_{s,t}$.
Since the middle-eastern reading follows the order $<_J$, 
we have 
\[
f^{-1}(a,b) = (i,j) <_J (s,t) = f^{-1}(a+1,b),
\]
which implies (d). Hence, both $f$ and $f^{-1}$ are PJ-standard and then
$f=\Psi(T)\in \Pic$. Now, we have completed the proof of Proposition 
\ref{prop:psi}. \qed

\renewcommand{\thesection}{\arabic{section}}
\section{Bijectivity of $\Phi$ and $\Psi$}
\setcounter{equation}{0}
\renewcommand{\theequation}{\thesection.\arabic{equation}}

In order to show 
$\Phi$ and $\Psi$ to be bijective, we shall prove
\[
(e)\,\,\Psi \circ \Phi = \text{id}_{\Pic}.\qq 
(f)\,\,\Phi \circ \Psi = \text{id}_{\Cry}.
\]

\noindent
(e) 
For $f=(f_1,f_2)\in \Pic$, set $g:=\Psi \circ \Phi(f)$.
$\Phi(f)$ is a Young tableau whose $(s,t)$-entry  
$\Phi(f)_{s,t}$ is equal to $f_1(s,t)$.
Let $m:=\Phi(f)_{s,t}$ be the $p$-th entry from the right in 
$\Phi(f)^{(m)}$ and then
\begin{equation*}
g(s,t) = (\Phi(f)_{s,t}, \lambda_{\Phi(f)_{s,t}}+p)
=(f_1(s,t), \lambda_{f_1(s,t)}+p).
\end{equation*}

We can easily see from Lemma~\ref{lem:add} that 
$f(s,t) = (\Phi(f)_{s,t} ,\lambda_{\Phi(f)_{s,t}}+p)
=(f_1(s,t), \lambda_{f_1(s,t)}+p)$.
Hence, we have $g=f$ and then $\Psi \circ \Phi =$ id$_\Pic$.

\vspace{4mm}
\noindent
(f) 
Take $T\in \Cry$. 
By the definition of $\Psi$, $\Psi(T)$ is a map which 
sends $(i, j)$ to $(T_{i,j},  \nu_{T_{i,j}}+p)$, 
where $p=p(T;i,j)$.
Furthermore, by the definition of $\Phi$, 
$\Phi \circ \Psi(T)$ is a Young tableau in the shape $\mu$ 
with a entry $T_{i,j}$
in a box $(i, j)$.
This means $ T = \Phi \circ \Psi(T)$ and then 
$\Phi \circ \Psi= \text{id}_{\Cry}$. 

Now, we have completed the proof of Theorem \ref{main}.\qed

\begin{ex}
Set 
$f:=$
\begin{tabular}{c||c|c|c|c|c}
$\mu$ & $(1,1)$ & $(1,2)$ & $(1,3)$ & $(2,1)$ & $(2,2)$\\
\hline
$\nu \setminus \lambda$ & $(1,4)$ 
& $(2,2)$ & $(4,1)$ & $(2,3)$ & $(3,2)$\\
\end{tabular}\\
$\in \Pic$.
We have 
$\Phi(f) = \threetwo(1,2,4,2,3)$.
Let us apply $\Psi$ to this. The number of entries 
$1,3,4$ in $\Phi(f)$ is one and then their destinations are
determined uniquely:
$1\mapsto (1,4)$, 
$3\mapsto(3,2)$ and $4\mapsto(4,1)$. 
There two entries $2$ in $\Phi(f)$. 
Since $2$ in $(1,2)$ is right to the one in $(2,1)$, it goes to 
$(2,2)$ and the other goes to $(2,3)$.
Hence we have, 

$\Psi \circ \Phi(f) =$
\begin{tabular}{c||c|c|c|c|c}
$\mu$ & $(1,1)$ & $(1,2)$ & $(1,3)$ & $(2,1)$ & $(2,2)$\\
\hline
$\nu \setminus \lambda$ & $(1,4)$ & $(2,2)$ & $(4,1)$ 
& $(2,3)$ & (3,2)\\
\end{tabular}
$=f$.\\
This shows $\Psi \circ \Phi =$ id$_{\Pic}$.
\end{ex}

\begin{ex}
Set $T:= \threetwo(1,2,4,2,3) \in \Cry$. 
We have \\
$\Psi(T) =$
\begin{tabular}{c||c|c|c|c|c}
$\mu$ & $(1,1)$ & $(1,2)$ & $(1,3)$ & $(2,1)$ & $(2,2)$\\
\hline
$\nu \setminus \lambda$ & $(1,4)$ 
& $(2,2)$ & $(4,1)$ & $(2,3)$ & (3,2)\\
\end{tabular}.\\
By the definition of $\Phi$, we obtain:
$\Phi \circ \Psi(T) =\threetwo(1,2,4,2,3)$. Hence, 
$\Phi \circ \Psi= \text{id}_{\Cry}$

\end{ex}

\renewcommand{\thesection}{\arabic{section}}
\section{Conjecture}
\setcounter{equation}{0}
\renewcommand{\theequation}{\thesection.\arabic{equation}}

We define a total order on a subset $X$ 
in $\bbN\times \bbN$, called 
``{\it admissible order}'' and denoted by $<_A$.
\begin{df}
\begin{enumerate}
\item
A total order $<_A$ on $X\subset\bbN\times\bbN$ 
is called $admissible$
if it satisfies:
\[
\text{For any }(a,b),\,\,(c,d)\in X\text{ if }
a\leq c\text{ and }b\geq d\text{ then } (a,b)<_A (c,d).
\]
\item
For $X,Y\subset\bbN\times\bbN$ and a map $f:X\to Y$, if $f$ satisfies
that if $(a,b)<_P(c,d)$, then $f(a,b)<_A f(c,d)$ for 
any $(a,b),~(c,d)\in X$, then $f$ is called PA-standard.
\end{enumerate}
\end{df}
{\sl Remark.} Note that for fixed $X\subset\bbN\times\bbN$, there 
can be several admissible orders on $X$. For example, 
the order $<_J$ is one of admissible orders on $X$. 
If we define the total order $<_F$ by 
\[
 (a,b)<_F (c,d) \text{ iff }b>d, \text{ or }b=d \text{ and }a<c, 
\]
then this is also admissible.

Let $\lm,\mu,\nu$ be Young diagrams as above and 
$<_A$ (resp. $<_{A'}$) an admissible order on $\nu\setminus\lm$
(resp. $\mu$). 
Note that we do not assume $<_A~=~<_{A'}$.
We define a set $(A,A')$-pictures 
${\bf P}(\mu,\nu\setminus\lm:A,A')$ by 
\[
{\bf P}(\mu,\nu\setminus\lm:A,A')
:=\left\{f:\mu\to\nu\setminus \lm~|~\begin{tabular}{l}\hbox{$f$ is PA-standard and
 bijective,} \\ \hbox{and $f^{-1}$ is PA$'$-standard.}\end{tabular}
\right\}.
\]
\begin{df}
Let $A$ be an admissible order on a Young diagram $\mu$ with $|\mu|=N$.
For $T\in B(\mu)$, by reading the entries in $T$ according to $A$, 
we obtain the map 
\[
 R_A:B(\mu)\longrightarrow B^{\ot N}\q (T\mapsto
\fsq(i_1)\ot\cd\ot\fsq(i_N))),
\]
which is called an admissible reading associated with the order $A$.
It is known that the map $R_A$ is an embedding of crystals(\cite{HK}).
\end{df}
Here note that Theorem \ref{LRrule} is valid for an arbitrary reading 
$R_A$, that is, in (\ref{LR}) we can replace ${\rm FE}(T)$ with 
$R_A(T)$. Define 
\[
\mathbf{B}(\mu)_\lm^\nu[A]
:= 
\left\{T\in{\mathbf B}(\mu)|
\begin{array}{l} 
R_A(T)=\MECry, \\
\text{for any }k=1,\cd,N,\\
\lm[i_1,\cd,i_k]\text{ is a Young diagram and}\\
\lm[i_1,\cd,i_N]=\nu.
\end{array}
\right\},
\]
It is shown in \cite{HK} that for any admissible order on $\mu$,
\begin{equation}
\mathbf{B}(\mu)_\lm^\nu[A]=\Cry.
\label{crya=cry}
\end{equation}
\begin{con}
Let $A$ (resp. $A'$) be an admissible order on 
$\nu\setminus\lm$ (resp. $\mu$). 
There exists a bijection 
\[
 \Psi:\mathbf{B}(\mu)_\lm^\nu[A']\longrightarrow
{\bf P}(\mu,\nu\setminus\lm:A,A'),
\]
where $\Psi$ is the same as in 4.1. 
\end{con} 
If we show the conjecture, 
together with (\ref{crya=cry}), we have 
\begin{cor}
For arbitrary admissible orders $A$ on $\nu\setminus\lm$ and 
$A'$ on $\mu$,
\[
\Pic= {\bf P}(\mu,\nu\setminus\lm:A,A').
\]
\end{cor}
This has been shown in \cite{CS2} and \cite{FG} by some purely combinatorial
way.

\end{document}